\documentclass[a4paper,11pt]{article}

\usepackage{graphicx}
\usepackage[affil-it]{authblk}

\usepackage{alltt} 

\usepackage{amsfonts}
\usepackage{amsmath}
\usepackage{amssymb}
\usepackage{amsthm}

\usepackage[all]{xy}

\usepackage{lipsum}

\newcommand\blfootnote[1]{%
  \begingroup
  \renewcommand\thefootnote{}\footnote{#1}%
  \addtocounter{footnote}{-1}%
  \endgroup
}

\newcommand{\RR}{{\mathbb R}}

\newcommand{\ZZ}{\mathbb Z}
\newcommand{\QQ}{{\mathbb Q}}

\def\G{{\mathcal G}}

\def\M{{\mathcal M}}

\def\P{{\mathcal P}}

\def\T{{\mathcal T}}

\numberwithin{equation}{section}

\newtheorem{teo}{Theorem}
\newtheorem{theo}{Theorem}
\newtheorem{proposition}[theo]{Proposition}
\newtheorem{coro}[theo]{Corollary}
\newtheorem{lemma}[theo]{Lemma}
\newtheorem{defi}[theo]{Definition}
\newtheorem{remark}[theo]{Remark}
 
\newtheorem{prop}[theo]{Proposition}
\author{Paulina Cecchi%
  \thanks{The first author was supported by the PhD grant CONICYT-PCHA/Doctorado Nacional/2015-21150544.}}
\affil{Departamento de Matem\'atica y Ciencia de la Computaci\'on, Universidad de Santiago de Chile  \\ paulina.cecchi@usach.cl}

\author{Mar\'{\i}a Isabel Cortez%
  \thanks{The research of the second author was supported by  Fondecyt Research Project
1140213.}}
\affil{Departamento de Matem\'atica y Ciencia de la Computaci\'on, Universidad de Santiago de Chile \\  maria.cortez@usach.cl}

\begin{document}

\title{Invariant measures for actions of congruent monotileable amenable groups. }

\date{Dated: \today}

\maketitle{}

\begin{abstract}    \noindent In this paper we  show that for every congruent monotileable amenable group $G$ and  for every metrizable Choquet simplex $K$, there exists a minimal $G$-subshift, which is free on a full measure set,  whose set of invariant probability measures is affine homeomorphic to $K$.  If the group is virtually abelian, the subshift is free. Congruent monotileable amenable groups are a generalization of amenable residually finite groups. In particular, we show that this class contains  all the infinite countable virtually nilpotent groups.  This article is a  generalization to congruent monotileable amenable groups of one of the principal results shown in \cite{CP} for residually finite groups. 

\end{abstract}

 \blfootnote{\textup{2010} \textit{Mathematics Subject Classification}: \textup{37B05, 37C85, 37B10}}

\section{Introduction}

\noindent The set of invariant probability measures of a continuous action of an amenable group on a compact metric space is a (non empty) metrizable Choquet simplex (see for example \cite{Gl}). A natural question is whether the converse is true, i.e,  if given a metrizable Choquet simplex $K$ and an amenable group $G$, it is   possible to realize $K$ as the set of invariant probability measures of a continuous action of $G$ on a compact metric space.   In \cite{D91}  Downarowicz answered  for the first time this question in the case $G=\ZZ$, showing  that every metrizable Choquet simplex  can be realized as the set of invariant probability measures of a Toeplitz $\ZZ$-subshift. The extension of this result to any amenable residually finite group $G$ was shown in \cite{CP}. In the recent work \cite{FH16}, the authors show that every face in the simplex of invariant measures of a zero-dimensional free dynamical system given by an action of an amenable group, can be realized as the entire simplex of invariant measures on some other zero-dimensional dynamical system with a free action of the same group. Thus, if the Poulsen simplex could be obtained as the set of invariant measures of a free action of some prescribed amenable group $G$, the same would hold for any Choquet simplex.\\  
 
\noindent The realization of Choquet simplices as sets of invariant measures is related to orbit equivalence problems. Indeed, an invariant for topological orbit equivalence among minimal free actions on the Cantor set is the ordered group with unit associated to those systems (which is a complete invariant for $\ZZ^d$-actions \cite{GMPS10,GPS95}).  On the other hand,  the space of traces of the associated ordered group with unit is affine homeomorphic to the set of invariant probability measures of the system (see for example \cite{HPS}).  Thus, an amenable group $G$ with the property that every metrizable Choquet simplex can be realized as the set of invariant probability measures of a minimal free $G$-action on the Cantor set, admits  at least as many  topological orbit equivalence classes as metrizable Choquet simplices exist.\\ 

\noindent In this paper we deal with the problem of realization of Choquet simplices in the context of actions of monotileable groups.  In \cite{W} Weiss introduces the concept of amenable monotileable group, which are a generalization of amenable residually finite groups, in the sense that, roughly speaking, the monotiles used to tile a monotileable group, play the role of the fundamental  domains of the finite index subgroups of the residually finite groups. It is still unknown if there are amenable groups which are not monotileable.  In this article, we introduce the concept of congruent monotileable amenable group, which include all the amenable residually finite groups. We show that the class of congruent monotileable amenable groups  is larger than the class  of amenable residually finite group.  More precisely, we show the following result.

\begin{teo}\label{nilpotentes}
Every countable virtually nilpotent group is congruent monotileable.
\end{teo} 
Thus abelian groups which are not residually finite (for example $\QQ$, or the Pr$\ddot{u}$fer group) are congruent monotileable.  It is an open question whether every monotileable group is congruent monotilable.

We answer the question of realization of Choquet simplices for congruent monotileable groups. Our principal result is the following. 

\begin{teo}\label{theoremA}
Let $G$ be an infinite congruent monotileable amenable group. For every metrizable Choquet simplex $K$, there  exists a minimal   $G$-subshift, which is free on a full measure set,   whose set of invariant probability measures is affine homeomorphic to $K$. If $G$ is virtually abelian, the subshift  is free.
\end{teo}
  
Combining Theorem \ref{nilpotentes} and Theorem \ref{theoremA}, we get that for any countable abelian group $G$ (even those which are not residually finite) and any metrizable Choquet simplex $K$ there exists a minimal free action on the Cantor set whose set of invariant probability measures is affine homeomorphic to $K$. 

It is important to remark that  when $G$ is not residually finite, the minimal free $G$-subshifts that we get are not Toeplitz. This is because Toeplitz subshifts are regularly recurrent, and the only groups admitting regularly recurrent actions on the Cantor set are residually finite groups. \\
  
\noindent This paper is organized as follows. In Section \ref{section-monotilable} we introduce the concept of congruent monotileable amenable group, and we show Theorem \ref{nilpotentes}. 
In Section  \ref{section-construction} we construct minimal  $G$-subshifts for any congruent monotileable amenable group $G$. We give a characterization of the set of invariant probability measures of those subshifts in terms of inverse limits as it was done in \cite{CP}. We also make some comments about the associated ordered group with unit, noting that its rational subgroup is non cyclic.  In Section \ref{section-results} we use the construction of the previous section to show our principal result.

\section{Monotileable amenable groups}\label{section-monotilable}
 
Through this paper,  $G$ is an amenable discrete infinite countable group. We denote by $1_G$ the unit of $G$. 

By a {\it right F\o lner sequence} of $G$ we mean a sequence $(F_n)_{n\geq 0}$ of finite subsets of $G$ such that for every $g\in G$,
 $$
  \lim_{n\to \infty}\frac{|F_ng\setminus F_n|}{|F_n|}=0.
 $$
We take the existence of a right F\o lner sequence as the definition of amenablitiy for $G$, that is, the group $G$  is {\it amenable} if it has a right F\o lner sequence.
 
 The next result is direct (see for example \cite[Proposition 4.7.1]{CC}).   
\begin{lemma}\label{remark-folner}
 The following conditions are equivalent:   
\begin{enumerate}
\item $(F_n)_{n\geq 0}$ is a  right F\o lner sequence of $G$.

\item For every $g\in G$,
$$
 \lim_{n\to \infty}\frac{|F_n\setminus F_ng|}{|F_n|}=0.
$$ 
\item $(F_n^{-1})_{n\geq 0}$ is a left F\o lner sequence of $G$, i.e, for every $g\in G$,
$$
  \lim_{n\to \infty}\frac{|gF_n\setminus F_n|}{|F_n|}=0.
 $$
\end{enumerate}   
 \end{lemma}
The conditions above are usually refered as the F\o lner conditions. Right F\o lner sequences are somehow those which become more and more invariant under right multiplication. We precisize this notion in the following lemma. Its proof is standard, we include it here for completeness.\\
 
\begin{lemma}\label{invariance}
$(F_n)_{n\geq 0}$ is a  right F\o lner sequence of $G$ if and only if for every finite subset $K$ of $G$ and for every $\varepsilon>0$,  there exists $N\geq 0$ such that for all $n\geq N$, $F_n$ is right $(K,\varepsilon)$-invariant, that is 
$$
\frac{|\{g\in F_n: gK\subset F_n\}|}{|F_n|}\geq (1-\varepsilon)
$$
\end{lemma}
\begin{proof}
First note that
$$\{g\in F_n: gK\subset F_n\}=\bigcap_{k\in K}F_n\cap F_nk^{-1}$$
Suppose $(F_n)_{n\geq 0}$ is a right F\o lner sequence, and take any finite $K\subseteq G$, $\varepsilon >0$. Since for all $k\in K$, $\lim_{n\to \infty}\frac{|F_n\setminus F_nk^{-1}|}{|F_n|}=0$, then for all $k\in K$, $\exists$ $N_k\in\mathbb{N}$ such that for all $n\geq N_k$,
$$\frac{|F_n\setminus F_nk^{-1}|}{|F_n|}<\frac{\varepsilon}{|K|}$$
Since $K$ is finte, there is a $N\in \mathbb{N}$ such that for all $n\geq N$ and for all $k\in K$, the above inequality holds. Now,
$$F_n\setminus \bigcap_{k\in K}F_n\cap F_nk^{-1}=\bigcup_{k\in K}F_n\setminus F_nk^{-1}$$
then,
\begin{align*}
|F_n|-|\bigcap_{k\in K}F_n\cap F_nk^{-1}|=& |F_n\setminus\bigcap_{k\in K}F_n\cap F_nk^{-1}|\\
=&|\bigcup_{k\in K}F_n\setminus F_nk^{-1}|\\
\leq & \sum_{k\in K}|F_n\setminus F_nk^{-1}|<|F_n|\varepsilon
\end{align*} 
and therefore, 
$$|F_n|-|\{g\in F_n: gK\subset F_n\}|<\varepsilon|F_n|$$
Conversely, suppose that for any $\varepsilon>0$ and any finite $K\subseteq G$, there is a $N\in \mathbb{N}$ such that for all $n\geq N$, $F_n$ is right $(K,\varepsilon)$-invariant. Take $g\in G$ and consider $K=\{g\}$. Let $\varepsilon>0$. Then, there is a $N\in \mathbb{N}$ such that for all $n\geq N$, 
$$|F_n\cap F_ng^{-1}|>(1-\varepsilon)|F_n|$$
$$\Leftrightarrow \forall n\geq N,\quad \frac{|F_n\setminus F_n\cap F_ng^{-1}|}{|F_n|}<\varepsilon$$
But $F_n\setminus F_n\cap F_ng^{-1}=F_n\setminus F_ng^{-1}$, and therefore, for all $n\geq N$,
$$\frac{|F_n\setminus F_ng^{-1}|}{|F_n|}<\varepsilon$$
This implies that $\lim_{n\to \infty}\frac{|F_ng\setminus F_n|}{|F_n|}\leq \varepsilon$. Since $\varepsilon$ was arbitrarily taken, we conclude that $\lim_{n\to \infty}\frac{|F_ng\setminus F_n|}{|F_n|}=0$.
\end{proof}

We will use the concept of \textit{left} $(K,\varepsilon)$-invariance as well: for finite subsets $F,K\subseteq G$ and $\varepsilon>0$, we say that $F$ is left $(K,\varepsilon)$-invariant if 
$$\frac{|\{g\in F: Kg\subset F\}|}{|F|}\geq (1-\varepsilon)$$
This is the notion of invariance used in \cite{W}. Note that if $F$ is right $(K,\varepsilon)$-invariant, then $F^{-1}$ is left $(K^{-1},\varepsilon)$-invariant.
\begin{remark}
 {\rm  Every abelian group is amenable. In this case, every F\o lner sequence is left and right at the same time, and the notions of right and left $(K,\varepsilon)$-invariance coincide. }
 \end{remark}
A {\it left monotile} of $G$ is a finite subset $F$ of $G$ for which there exists a subset $C$ of  $G$ such that the collection $\T=\{cF: c\in C\}$ is a partition of $G$. In this case we say that $\T$ is a {\it left monotiling}. A {\it right monotile} of $G$ is a finite subset $F$ of $G$ for which there exists a subset $C$ of  $G$ such that the collection $\T=\{Fc: c\in C\}$ is a partition of $G$. In this case we say that $\T$ is a {\it right monotiling}. 
We say that $G$ is  {\it monotileable amenable} if  there exists a right F\o lner sequence $(F_n)_{n\geq 0}$ of $G$ such that every $F_n$ is a left monotile of $G$. Note that if $G$ is monotilable amenable, then $(F_n^{-1})_{n\geq 0}$ is a left F\o lner sequence whose elements are right monotiles of $G$.

\subsection{Congruent monotileable amenable groups.}

 \begin{defi}\label{defi2} Let $(F_n)_{n\geq 0}$ be a sequence of finite subsets of G. We say that  the sequence is congruent if $1_G\in F_0$ and for every $n\geq 0$ the exists a set $J_n\subseteq G$ such that $1_G\in J_n$ and  such that $\{cF_{n}: c\in J_n\}$ is a partition of $F_{n+1}$.\\
We say that  $G$ is congruent monotileable if it admits a congruent right F\o lner sequence made of left monotiles, which is also \textit{exhaustive}, that is, $G=\bigcup_{n\geq 0} F_n$.
  \end{defi}

\begin{lemma}\label{lema1} Let $G$ be a congruent monotilable amenable group with a right F\o lner sequence $(F_n)_{n\geq 0}$ made of congruent left monotiles.   Then  
for every  $m>n\geq 0$,  the collection 
$$
 \{c_{m-1}\cdots c_nF_n: c_i\in J_i, \mbox{ for every } n\leq i <m\}
$$
is a partition of $F_m$.
\end{lemma}
\begin{proof}  The proof follows directly from Definition \ref{defi2} by using induction.
\end{proof}

The next lemma is the key tool to show that countable abelian and nilpotent groups are congruent monotilable. If $M$ is any set, $\sim$ is an equivalence relation on $M$, $\pi:M\to M/\sim$ is the canonical projection and $H\subseteq M/\sim$, we say that $\widehat{H}\subseteq M$ is a \textit{lifting} of $H$ if $\pi(\widehat{H})=H$ and $\pi$ is one-to-one in $\widehat{H}$. 

\begin{lemma}\label{key-lemma}
Let $L$, $G$ and $Q$ be countable discrete amenable groups such that $1\to L\to G\to Q\to 1$ is an exact sequence. Suppose $L$ and $Q$ have congruent and exhaustive right F\o lner sequences made of left monotiles, $(U_s)_{s\geq 0}$ and $(T_s)_{s\geq 0}$ respectively.  Then, $G$ has an exhaustive right F\o lner sequence made of left monotiles. More precisely, there exists a sequence $(\widehat{T}_s)_{s\geq 0}$, such that each $\widehat{T}_s\subseteq G$ is a lifting of $T_s$, and an increasing sequence of indices, $(m_s)_{s\geq 0}$, such that $F_s=U_{m_s}\widehat{T}_s$ defines an exhaustive right F\o lner sequence made of letf monotiles of G. If in adittion $L\subseteq Z(G)$, then $(F_s)_{s\geq 0}$ is also congruent.
\end{lemma}

\begin{proof}
Let $(K_s)_{s\geq 0}$ be an increasing sequence of finite subsets of $G$ such that $G=\bigcup_{s\geq 0}K_s$. Let $(\varepsilon_s)_{s\geq 0}$ be a decreasing to zero sequence of positive reals. Let $\pi:G\to G/L\cong Q$ be the projection of $G$ on $G/L$. Since $(T_s)_{s\geq 0}$ is a right F\o lner sequence, lemma \ref{invariance} tells us that, up to take an increasing subsequence, we can assume that for all $s\geq 0$, $T_s$ is right $(\pi(K_s),\varepsilon_s/2)$-invariant. \\
We proceed by induction for de definition of $(\widehat{T}_s)_{s\geq 0}$ and $(m_s)_{s\geq 0}$. For $s=0$, define $\widehat{T}_0=1_G$ and $m_0=0$, so that $1_G\in F_0=U_0$.\\
If $s>0$, suppose we have defined $\widehat{T}_{s-1}$ and $m_{s-1}$. By hypothesis, there exists a finite subset $D\subseteq Q$ containing $1_Q$ such that 
$$T_s=\bigsqcup_{d\in D}dT_{s-1}$$
Take any lifting $\widehat{D}$ of $D$, and define $\widehat{T}_s=\bigcup_{\hat{d}\in \widehat{D}}\hat{d}\widehat{T}_{s-1}$. It is straightforward to show that $\widehat{T}_s$ is a lifting of $T_s$ and that   
$\widehat{T}_s=\bigsqcup_{\hat{d}\in \widehat{D}}\hat{d}\widehat{T}_{s-1}$.\\
Note that $T_s^{-1}$ is a left $(\pi(K_s^{-1}),\varepsilon_s/2)$-invariant right monotile of $Q$ and that $\widehat{T}_s^{-1}$ is a lifting of $T_s^{-1}$. Then, by \cite{W}, Theorem 2, there exists a subset $J\subseteq L$ such that, if $U\subseteq L$ is a left $(J,\varepsilon_s/2)$-invariant right monotile of $L$, then $\widehat{T}_s^{-1}U$ is a left $(K_s^{-1},\varepsilon_s)$-invariant right monotile of $G$.\\
Now, every $U_s$ is a left monotile of $L$ and therefore every $U_s^{-1}$ is a right monotile of $L$. Since $(U_s^{-1})_{s\geq 0}$ is a letf F\o lner sequence, the $U_s^{-1}$'s are as much left invariant as we want. Pick an index $m_s>m_{s-1}$ such that $U_{m_s}^{-1}$ is left $(J,\varepsilon_s/2)$-invariant. By \cite{W}, Theorem 2, $\widehat{T}_s^{-1}U_{m_s}^{-1}$ is a left $(K_s^{-1},\varepsilon_s)$-invariant right monotile of $G$, and therefore $F_s:=U_{m_s}\widehat{T}_s$ is a right $(K_s,\varepsilon_s)$-invariant left monotile of $G$. This shows that $(F_s)_{s\geq 0}$ is a right F\o lner sequence made  of left monotiles of $G$.\\
Let us show that $(F_s)_{s\geq 0}$ is exhaustive. Let $g\in G$, then there is a $s_0$ such that $\pi(g)\in T_{s_0}$, since $(T_s)_{s\geq 0}$ is exhaustive. Then, there exist $l\in L$ and $\hat{t}\in \widehat{T}_{s_0}$ such that $g=l\hat{t}$. For $s\geq 0$ big enough, $l\in U_{m_s}$ and $\hat{t}\in \widehat{T}_s$, therefore $g\in F_s$.\\
Finally, let us see that, if $L\subseteq Z(G)$, $(F_s)_{s\geq 0}$ is congruent. Let $s\in \mathbb{N}$. Since $(m_s)_{s\geq 0}$ is increasing and $(U_s)_{s\geq 0}$ is congruent, lemma \ref{lema1} tells us that there exists a finite subset of $L$, $C$, such that
$$U_{m_{s+1}}=\bigsqcup_{c\in C}cU_{m_s}$$
On the other hand, by construction, there is a set $\widehat{E}\subseteq G$ such that
$$\widehat{T}_{s+1}=\bigsqcup_{\hat{e}\in \widehat{E}} \hat{e}\widehat{T}_s$$
Therefore,
\begin{eqnarray*}
  F_{s+1}  &=& \bigsqcup_{c\in C}cU_{m_s}\bigsqcup_{\hat{e}\in\widehat{E}}\hat{e}\widehat{T}_s \\
           &=& \bigsqcup_{c\in C}\bigsqcup_{\hat{e}\in\widehat{E}}cU_{m_s}\hat{e}\widehat{T}_s\\
           &=& \bigsqcup_{c\in C}\bigsqcup_{\hat{e}\in\widehat{E}}c\hat{e}\underbrace{U_{m_s}\widehat{T}_s}_{=F_s}\\
\end{eqnarray*}
  
\end{proof}

\begin{remark}\label{virtually}
{\rm If $G$ is a countable amenable group having a finite index subgroup which is congruent monotilable, then so is $G$. Indeed, it is not difficult to see that, if $L$ is a congruent monotilable finite index subgroup of $G$, $R$ is a subset of right coset representatives, and $(U_n)_{n\geq 0}$ is a congruent right F{\o}lner sequence made of left monotiles of $L$, then the sequence $F_n:=U_nR$ defines a congruent right F{\o}lner sequence of $G$. This implies that any virtually congruent monotilable group is congruent monotilable as well. }
\end{remark}

\begin{prop}\label{abelianos}
Every countable abelian group is congruent monotilable.
\end{prop} 
\begin{proof} We concentrate in the case of infinite countable abelian groups, since the finite case is trivial. Recall that if $G$ is a finitely generated abelian group, then it has polinomial growth and therefore it is residually finite. Any such a group (provided it is amenable) is congruent monotilable (see for example \cite[Lemma 4]{CP}).\\

\noindent Let $G$ be a non-finitely generated abelian group, and enumerate its elements  as  $G=\{1_G=g_0,g_1,\cdots\}$. Since $G$ is non-finitely generated we can define an  increasing sequence $(k_n)_{n\geq 0}$ as follows:  let $k_0=0$ and for $n>0$ let define 
    $$
    k_{n}=\min\{l>k_{n-1}: g_{k_n}\notin \langle\{g_0,\cdots, g_{k_{n-1}}\}\rangle\}.
    $$
For every $n\geq 0$ we set  
$$K_n=\{g_0,g_1,\cdots,g_{k_n}\}  \mbox{ and } G_n=\langle K_n \rangle,$$
where $\langle \cdot \rangle$ denotes the generated subgroup in $G$.

Since $G$ is abelian, $G_{n-1}\lhd G_n$ and  $G_n/G_{n-1}$ is an abelian group. Moreover, this is a non trivial cyclic group.  Indeed, any class $gG_{n-1}$ has the form
$$gG_{n-1}=g_0^{l_0}\cdots g_{k_n-1}^{l_{k_n-1}}g_{k_n}^{l_{k_n}}G_{n-1} \mbox{ where } l_i\in \mathbb{Z}.$$
Since $g_0^{l_0}\cdots g_{k_n-1}^{l_{k_{n}-1}} \in G_{n-1}$ and $G_n$ is abelian, $gG_{n-1}=g_n^{l_{k_n}}G_{n-1}\neq G_{n-1}$, so that 
$$G_n/G_{n-1}=\{g_{k_n}^kG_{n-1}:k\in\mathbb{Z}\}.$$
If for all $k\neq 0$ $g_{k_n}^k\notin G_{n-1}$, then $G_n/G_{n-1}\cong \mathbb{Z}$. If there is some $k\neq 0$ such that $g_{k_n}^k\in G_{n-1}$, then $G_n/G_{n-1}\cong \mathbb{Z}/l\mathbb{Z}$, where 
$$l:=\min\{k\in\mathbb{Z}^+:g_n^k\in G_{n-1}\}.$$
The rest of the proof is organized as follows:
\begin{itemize}
\item For every $n\geq 0$, we will inductively define $(F_s^n)_{s\geq 0}$ an exhaustive and congruent right F\o lner sequence made of left monotiles of $G_n$, a positive integer $m_{n-1}$ and a finite subset $F_{n-1}\subseteq G$.
\item  We will show that  $(F^n_{m_{n}})_{n\geq 0}$ is an exhaustive and congruent right F\o lner sequence of left monotiles of $G$.  
\end{itemize}
For $n=0$, define $F_s^{0}=\{1_G\}$, for every $s\geq 0$. This is an exhaustive and congruent right F\o lner sequence of left monotiles of $G_0$, because this group is trivial.  We also put $m_{-1}=0$ and $F_{-1}=\{1_G\}$.  \\
  
\noindent Let  $n>0$. Suppose we have defined an exhaustive and congruent right F\o lner sequence $(F^{n-1}_s)_{s\geq 0}$  of left monotiles of $G_{n-1}$, the positive integer $m_{n-2}$ and the subset $F_{n-2}$. Because $(F^{n-1}_s)_{s\geq 0}$ is exhaustive and $K_{n-1}$ does not depend on the parameter $s$,  we can assume  that $K_{n-1} \subseteq F^{n-1}_s$ and $F^{n-1}_s$ is right $(K_{n-1}, \varepsilon_{n-1})$-invariant, for every $s\geq 0$.\\   

\noindent Since $G_n/G_{n-1}$ is cyclic (and then residually finite), it admits an exhaustive and congruent right F\o lner sequence of left monotiles. Let us denote this sequence as  $(T^{n-1}_s)_{s\geq 0}$.\\
 
\noindent From Lemma \ref{key-lemma}, there exist a sequence of liftings, $(\widehat{T}_s^{n-1})$ and an increasing sequence of indices, $(m_{n-1,s})_{s\geq 0}$, such that $(F^{n-1}_{m_{n-1,s}}\widehat{T}^{n-1}_s)_{s\geq 0}$ is an exhaustive congruent right F\o lner sequence made of left monotiles of $G_n$.  We define
 $$
  F_s^n=F^{n-1}_{m_{n-1,s}}\widehat{T}^{n-1}_s \mbox{ for every } s\geq 0.
 $$
  
We can assume that for every $s\geq 0$, $m_{n-1,s}>m_{n-2}$, $K_n\subseteq F_s^n$ and that $F_s^n$ is right $(K_n, \varepsilon_n)$-invariant. We define $m_{n-1}=m_{n-1,0}$, and
$$
  F_{n-1}=F^{n-1}_{m_{n-1}}.
$$

\underline{Claim:} $(F_n)_{n\geq 0}$ is an exhaustive congruent right F\o lner sequence of $G$ made of left monotiles.\\
This sequence is right F\o lner because for all $n\geq 0$, $F_n=F^{n}_{m_{n,0}}$ is right $(K_{n}, \varepsilon_{n})$-invariant. 

Since $F_n=F^n_{m_{n,0}}$ is a left monotile of $G_n$, we have that $F_n$ is a left monotile of $G$. Indeed, if $C_n\subseteq G_n$ is such that $\{cF_n: c\in C_n\}$ is a partition of $G_n$, and $\Lambda_n\subseteq G$ is a lifting of $G/G_n$, then $\{gcF_n: g\in \Lambda_n, c\in C_n\}$ is a partition of $G$.
 
 We have that $F_n=F_{m_n}^n=F^{n}_{m_{n-1,m_n}}$. Since $(m_{n-1,s})_{s\geq 0}$ is increasing, and $(F_s^n)_{s\geq 0}$ is congruent, every $F_{m_{n-1,s}}^{n-1}$ is a disjoint union of translated copies of  $F_{m_{n-1,0}}^{n-1}=F_{n-1}$.  This together with the fact that  all the elements in  $\widehat{T}^{n-1}_{m_n}$ are  in diferent classes of $G_n/G_{n-1}$, imply that $F_n$ is a disjoint union of translated copies of $F_{n-1}$. 
  \end{proof}

\begin{proof}[Proof of Theorem 1] Suppose $G$ is a countable nilpotent group. Again recall that if $G$ is finitely generated, it is residually finite and then congruent monotilable by the results in \cite{CP}. If $G$ is not finitely generated, we use induction on the nilpotency class of the group. If $G$ is of class $1$, then $G$ is abelian and the result follows from proposition \ref{abelianos}. If its nilpotency class is $2$, then consider the exact sequence
$$1\to[G,G]\to G\to G/[G,G]\to 1$$
where $[G,G]$ denotes the commutator subgroup of $G$. Since both $[G,G]$ and $G/[G,G]$ are abelian, each has a congruent right F{\o}lner sequence made of left monotiles. Since $[G,G]\leqslant Z(G)$, we deduce from Lemma \ref{key-lemma} that $G$ is congruent monotilable.\\
If $G$ has nilpotency class $n$ grater than $2$, consider the exact sequence
$$1\to G^{n-1}\to G\to G/G^{n-1}\to 1$$ 
where $G^i$ denotes the $i$-th subgroup in the lower central series of $G$. Since $G^{n-1}\leqslant Z(G)$, it is abelian, and therefore it follows from proposition \ref{abelianos} that $G^{n-1}$ is congruent monotilable. On the other hand, $G/G^{n-1}$ is a group of nilpotency class $n-1$ (this follows from the fact that for all $0\leq i\leq n-1$, $(G/G^{n-1})^i=G^i/G^{n-1}$), and then by inductive hypothesis it is congruent monotilable as well. Since $G^{n-1}\leqslant Z(G)$, we deduce from Lemma \ref{key-lemma} that $G$ is congruent monotilable. This proves that any countable nilpotent group is congruent monotilable.\\ 
From the argument above and remark \ref{virtually}, we deduce that every countable virtually nilpotent group is congruent monotilable. 
\end{proof}


\section{Invariant measures  and group  actions on the Cantor set.}\label{section-construction}
In this section we assume that $G$ is a congruent monotilable  amenable group with a right F\o lner sequence $(F_n)_{n\geq 0}$ made of congruent left monotiles.  We use the notation of Definition \ref{defi2}.

\subsection{Subshifts.}
Let $\Sigma$ be a finite alphabet.  By $\Sigma^G$ we mean the set of all the
functions from $G$ to $\Sigma$. The (left) {\it shift action} $\sigma$
of $G$ on $\Sigma^{G}$ is given by  $\sigma^g(x)(h)=x(g^{-1}h)$, for every $h\in G$ and $x\in \Sigma^G$.
We consider $\Sigma$ endowed with the discrete topology and
$\Sigma^{G}$ with the product topology. Thus every $\sigma^g$ is a
homeomorphism of the Cantor set $\Sigma^G$.  A {\it subshift } or $G$-subshift of $\Sigma^{G}$ is a closed subset
$X$ of $\Sigma^{G}$ which is invariant by the shift action.
The triple $(X,\sigma|_X, G)$ is also called
subshift or $G$-subshift. See \cite{CC} for more details.

The subshift $(X,\sigma|_X, G)$ is  {\it minimal} if  for every $x\in X$  the orbit
$o_{\sigma}(x)=\{\sigma^g(x): g\in G\}$ is dense in $X$. We say that $(X,\sigma|_X, G)$ is {\it free} on $A\subseteq X$ if $\sigma^g(x)=x$ implies $g=1_G$, for every $x\in A$.  If $A=X$ we just say that the subshift  is free.

\subsection{Construction of a minimal  $G$-subshift $(X, \sigma|_X, G)$.  }

\begin{remark}{\rm Here we use the principal ideas introduced in  \cite{CP}. Nevertheless, unlike \cite{CP}, in the subshifts of this article  there are not  finite index subgroups of $G$ as set of return times to clopen sets. This will change the strategies of the proofs and the properties of the  $G$-subshifts that we construct. }
\end{remark}

Let $k_0\geq 3$ be an integer and let  $\Sigma=\{0,\cdots, k_0\}$. For every $1\leq k\leq k_0$ let define $B_{0,k}\in \Sigma^{F_0}$ as
$$
B_{0,k}(v)= \left\{\begin{array}{ll}
                                            k  &  \mbox{ if } v=1_G\\
                                            0  & \mbox{ if } v\in F_0\setminus\{1_G\}.
                                            \end{array}\right. 
$$ 

For $n\geq 0$, let $k_{n+1}\geq 3$ be an integer and let $B_{n+1,1},\cdots, B_{n+1,k_{n+1}}$ be different elements in $\Sigma^{F_{n+1}}$ verifying the following conditions:
\begin{itemize}
\item[(C1)] $B_{n+1,k}(F_n)=B_{n,1}$, for every $1\leq k\leq k_{n+1}$.

\item[(C2)] $B_{n+1,k}(cF_n)\in \{B_{n,2},\cdots, B_{n,k_n}\}$ for every $c\in J_n\setminus \{1_G\}$. 
\end{itemize}

\begin{lemma}\label{partition}
Let $(B_{n,1},\cdots, B_{n,k_n})_{n\geq 0}$ be  the sequence  defined above. Then for every $n\geq 0$ we have the following:
\begin{itemize}
\item[(C3)] If $g\in F_n$ and $1\leq k,k'\leq k_n$ are such that $B_{n,k}(gv)=B_{n,k'}(v)$ for every $v\in F_n\cap g^{-1}F_n$, then $g=1_G$ and $k=k'$.
\end{itemize}
\end{lemma}
\begin{proof}
The case $n=0$ is clear.

\medskip

Suppose the hypothesis is true for $n\geq 0$. Let $g\in F_{n+1}$ and $1\leq k,k'\leq k_{n+1}$ be such that
$$
 B_{n+1,k}(gv)=B_{n+1,k'}(v) \mbox{ for every } v\in F_{n+1}\cap g^{-1}F_{n+1}.
$$
Let $c_n\in J_n$ and $s\in F_n$ be such that $g=c_ns$. Since $c_nF_n\subseteq F_{n+1}$, conditions (C1) and (C2) imply
$$
B_{n,l}(su)=B_{n,1}(u) \mbox{ for every } u\in F_{n}\cap s^{-1}F_{n}, 
$$
where $1\leq l \leq k_n$ is such that $B_{n+1,k}(c_nF_n)=B_{n,l}$. By hypothesis we get $s=1$ and $l=1$. Conditions (C1) and (C2) imply $c_n=1_G$. From this we deduce $g=1_G$ and $k=k'$. 
\end{proof}

\begin{lemma}\label{unique-element}
Let $(B_{n,1},\cdots, B_{n,k_n})_{n\geq 0}$ be  the sequence  defined  above. Then 
  $$\bigcap_{n\geq 0}\{x\in \Sigma^G: x(F_n)=B_{n,1}\}=\{x_0\},$$ where $x_0$ is some element in $\Sigma^G$.
\end{lemma}
\begin{proof}
By Condition (C1) and because every set $\{x\in X: x(F_n)=B_{n,1}\}$ is compact, we have
 $$\bigcap_{n\geq 0}\{x\in \Sigma^G: x(F_n)=B_{n,1}\}\neq \emptyset.$$
Since the F\o lner sequence $(F_n)_{n\geq 0}$  is exhaustive, we deduce there exists only one element $x_0$ in this intersection.
\end{proof}
Let $x_0\in \Sigma^G$ be the element of Lemma \ref{unique-element}. Consider $X=\overline{\{\sigma^g(x_0): g\in G \}}$. For every $n\geq 0$ and $1\leq k\leq k_n$ we define
 $$
 C_{n,k}=\{x\in X: x(F_n)=B_{n,k}\} \mbox{ and } C_n=\bigcup_{k=1}^{k_n}C_{n,k}.
 $$

\begin{lemma}\label{segundo}
 Let $v\in G$. Then the following are equivalent:
 \begin{enumerate}
 \item $\sigma^{v^{-1}}(x_0)\in C_n$.
 \item  There exist $m > n$ and $c_i\in J_i$ for every $n\leq i<m$ such that\\   $v=c_{m-1}\cdots c_n$. \end{enumerate}
\end{lemma}  
\begin{proof}

Suppose that $v\in G$ is such that  $\sigma^{v^{-1}}(x_0)\in C_n$. If $v=1_G$ then for $m=n+1$ and $c_n=1_G$ we get the desired property. Suppose now that  $v\neq 1_G$. let $m\geq 0$ be the smallest integer such that $vF_n\subseteq F_{m}$. Because $|vF_n|=|F_n|$, it is necessary that $m\geq n$. Suppose that $m=n$. Since $1_G\in F_n$ this implies that $v\in F_n$. By hypothesis  we have
 $$
 X_0(vs)=B_{n,l}(s) \mbox{ for every } s\in F_n.
 $$ 
On the other hand,
$$
x_0(vs)=B_{n,1}(vs) \mbox{ for every } s\in v^{-1}F_n.
$$ 
Lemma \ref{partition} implies that $v=1_G$, which is a contradiction. Thus we have $m > n$. 
 
 \medskip
 
Since $vF_n\subseteq F_m$ and $1_G\in F_m$, we have $v\in F_m$. Lemma \ref{lema1}  implies for every $n\leq i\leq m-1$ there exist $c_i\in J_i$ such that  $v\in c_{m-1}\cdots c_nF_n$. Let $s\in F_n$ such that $v=c_{m-1}\cdots c_n s$.  By definition of $x_0$, we have
 $$
 x_0(c_{m-1}\cdots c_nF_n)=B_{n,k}, \mbox{ for some } 1\leq k\leq k_n,
 $$
which implies that for every $g\in s^{-1}F_n$,
$$
x_0(vg)=x_0(c_{m-1}\cdots c_n sg)=B_{n,k}(sg).
$$
On the other hand, for every $g\in F_n$ we have
$$
x_0(vg)=B_{n,l}(g).
$$
Thus 
$$
B_{n,l}(g)=B_{n,k}(sg) \mbox{ for every } g\in F_n\cap s^{-1}F_n,
$$
Lemma \ref{partition} implies that $s=1_G$ and then $v=c_{m-1}\cdots c_n$.

\medskip

 By induction and Lemma \ref{lema1}, we get that $v=c_{m-1}\cdots c_n$ is a return time of $x_0$ to $C_n$. 
 \end{proof}

\begin{lemma}
For every $n\geq 0$, consider the collection of sets  
 $$
 \P_n=\{\sigma^{v^{-1}}(C_{n,k}): v\in F_n, 1\leq k\leq k_n\},
 $$
 where $C_{n,k}$ is defined as above. Then for every $n\geq 0$,
 \begin{itemize}
 \item[(KR1)] $\P_n$ clopen partitions of  $X$.
 \item[(KR2)]  $\P_{n+1}$ is finer than $\P_n$.
\end{itemize}
 \end{lemma}
\begin{proof}

(KR1)  Suppose that $v,u\in F_n$ and $1\leq k,l\leq k_n$ are such that
$$
\sigma^{v^{-1}}(C_{n,k})\cap \sigma^{u^{-1}}(C_{n,l})\neq \emptyset.
$$ 
Then there exists $x\in X$ such that
$$
x\in \sigma^{uv^{-1}}(C_{n,k})\cap C_{n,l}. 
$$
Since $x\in X$, there exists a sequence $(g_i)_{i\geq 0}$ of elements in $G$ such that 
$$
\lim_{i\to\infty} \sigma^{g_i}(x_0)=x  \mbox{ and then } \lim_{i\to \infty}\sigma^{vu^{-1}g_i}(x_0)=\sigma^{vu^{-1}}(x).
$$ 
Lemma \ref{segundo} implies that for a large enough $i$ there exist  $c_p\in J_p$ for every $n \leq p\leq m-1$, and $\tilde{c}_p\in J_p$ for every   $n \leq p\leq r-1$,  such that 
$$
g_i^{-1}=c_{m-1}\cdots c_n \mbox{ and } g_i^{-1}vu^{-1}=\tilde{c}_{r-1}\cdots \tilde{c}_n,
$$
where $m\geq n+1$ and $r\geq n+1$ are the smallest integers such that $g_i^{-1}F_n\subseteq F_m$ and $g_i^{-1}vu^{-1}F_n\subseteq F_r$ respectively.

Then
$$
c_{m-1}\cdots c_nv=\tilde{c}_{r-1}\cdots \tilde{c}_nu.
$$
Suppose that $r\geq m$. From Lemma \ref{lema1}, this implies that 
$$
\tilde{c}_r=\cdots \tilde{c}_m=1_G  \mbox{ and } \tilde{c}_i=c_i \mbox{ for every } n\leq i\leq  m-1.
$$ 
We get that $u=v$ and then the sets in $\P_n$ are disjoint.

\medskip

Let  $g\in G\setminus F_n$. Let $m > n$ be such that $g\in F_m$. Then the congruency of $(F_n)_{n\geq 0}$ implies there exist $c_i\in J_i$ for every $n\leq i <m$ such that $g=c_{m-1}\cdots c_nu$, for some $u\in F_n$. Then from Lemma \ref{lema1} we get
$$
\sigma^{g^{-1}}(x_0)=\sigma^{u^{-1}}(\sigma^{(c_{m-1}\cdots c_n)^{-1}}(x_0))\in \sigma^{u^{-1}}(C_{n,l}), \mbox{ for some } 1\leq l\leq k_n.
$$ 
This shows that $\P_n$ is  a covering of $X$.

\medskip

(KR2) Condition (C2) implies that $\P_{n+1}$ is finer that $\P_n$.

 \end{proof}
 
Let
$$
\partial X= \bigcup_{g\in G}\bigcap_{n\geq 0} \bigcup_{k=1}^{k_n}\bigcup_{v\in F_n\setminus F_ng}\sigma^{v^{-1}}(C_{n,k}).
$$

 \begin{proposition}\label{minimal-free}
The system $(X, \sigma|_X, G)$ is minimal and free on $X\setminus \partial X$. If $G$ is virtually abelian, the system is free.
 \end{proposition}
\begin{proof}
Let $F\subseteq G$ a finite set, and let $P\in \Sigma^F$ such that
$$
C=\{x\in X: x(F)=P\}\neq \emptyset.
$$
We will show that  $R_C(x_0)=\{g\in G: \sigma^{g^{-1}}(x_0)\in C\}$ is syndetic, which is enough to conclude that the subshift is minimal (see for example \cite[Chapter 1]{A}).

Let $g\in G$ such that $\sigma^g(x_0)\in C$. Since the orbit of $x_0$ is dense in $X$ such a $g$ always exists.  We have $x_0(g^{-1}F)=P$. Let $n>0$ such that 
$g^{-1}F\subseteq F_{n-1}$. Then
$$
x_0(g^{-1}F)=B_{n-1,1}(g^{-1}F)=P.
$$ 
Condition (C1) and Lemma \ref{segundo} imply that for every $c_{m-1}\in J_{m-1}, \cdots, c_n\in J_n$, with $m>n$, we have
$$
\sigma^{(c_{m-1}\cdots c_n)^{-1}}(x_0)\in C_n\subseteq C_{n-1,1}.
$$
Thus we get
$$
\sigma^{(c_{m-1}\cdots c_n)^{-1}}(x_0)(g^{-1}F)=B_{n-1,1}(g^{-1}F)=P.
$$
This shows that $c_{m-1}\cdots c_ng^{-1}\in R_C(x_0)$.
Now let  $h\in G$  and  $m>n$ be such that $g\in F_m$. Lemma \ref{lema1} implies there are $c_{m-1}\in J_{m-1}, \cdots, c_n\in J_n$ such that $h\in c_{m-1}\cdots c_nF_n$. Then we get  $h\in R_{C}(x_0)gF_n$, which implies that $R_{C}(x_0)$ is  syndetic.

Let $x\in X$ and $g\in G$ be such that $\sigma^{g}(x)=x$. For every $n\geq 0$, let $v_n\in F_n$ be such that $x\in \sigma^{v_n^{-1}}(C_n)$. We have $x=\sigma^g(x)\in \sigma^{gv_n^{-1}}(C_n)$.  Because of $\P_n$ is a partition, if  there exists $n\geq 0$ such that $v_ng^{-1}\in F_n$, then $g=1_G$. Thus if there exists $g\in G\setminus \{1_G\}$ such that $\sigma^g(x)=x$, then $v_n\in F_n\setminus F_ng$ for every $n\geq 0$. This shows that the subshift is free on $X\setminus \partial X$.  

Suppose that $G$ is virtually abelian. Let  $\Gamma$ be an abelian finite index subgroup of $G$.  Because $G/\Gamma$ is finite, there exist $l>k\geq 1$ such that $g^l\Gamma=g^k\Gamma$, which implies $g^{l-k}\in \Gamma$.  Thus we can assume that $g\in \Gamma$. On the other hand, there exist a subsequence $(v_{n_k})_{k\geq 0}$ and $v\in G$ such that $v_{n_k}\in v\Gamma$, for every $k\geq 0$.  Let $\gamma_k\in \Gamma$ be such that $v_{n_k}=v\gamma_k$, for every $k\geq 0$.  Since $\sigma^{v_{n_k}}(x)\in C_{n_k}$, we have
$\lim_{k\to \infty}\sigma^{v^{-1}v_{n_k}}(x)=\lim_{k\to \infty}\sigma^{\gamma_{k}}(x)=\sigma^{v^{-1}}(x_0)$.  This implies that $\lim_{k\to \infty}\sigma^{g\gamma_{k}}(x)=\sigma^{gv^{-1}}(x_0)$ and because $\sigma^{g\gamma_{k}}(x)=\sigma^{\gamma_{k}g}(x)=\sigma^{\gamma_k}(x)$, we conclude that
$\sigma^{gv^{-1}}(x_0)=\sigma^{v^{-1}}(x_0)$.   Since $x_0\in X\setminus \partial X$, we deduce $gv^{-1}=v^{-1}$  and then the system is free. 
\end{proof}

\subsection{Invariant measures of $(X, \sigma|_X, G)$.}

A (metrizable) {\it Choquet simplex} is a compact, convex, and metrizable subset $K$ of a locally convex
real vector space, such that 
for each $v\in K$ there is a unique probability measure $\mu$
supported on the set of extreme points of $K$ such that $\int x
d\mu(x)=v$.

\medskip

An {\it invariant measure} of $(X,\sigma|_X, G)$ is a probability measure $\mu$ defined on the Borel sets of $X$ such that $\mu(\sigma^g(A))=\mu(A)$, for every $g\in G$ and every Borel subset $A$ of $X$. We denote $\M(X, \sigma|_X,G)$  the space of all the invariant measures of $(X,\sigma|_X, G)$. Because $G$ is amenable, this is a non empty Choquet simplex \cite{Gl}.  We say that $A\subseteq X$ is a {\it full measure} set of  $X$  if $A^c$ is negligible with respect to any invariant measure of $(X,\sigma|_X, G)$.

\subsubsection{Managed sequence of incidence matrices.}
For every $n\geq 0$ and  $1\leq i\leq k_n$, we define
$$
J_{n,k,i}=\{c\in J_n: B_{n+1,k}(cF_n)=B_{n,i}\},
$$
and $M_n\in \M_{k_n\times k_{n+1}}(\ZZ^+)$ as
$$
M_n(i,j)=|J_{n,j,i}|, \mbox{ for every } 1\leq i \leq k_n, 1\leq j\leq k_{n+1}. 
$$
 It is easy to see that
 $$
 M_n(i,j)=|\{v\in F_{n+1}: \sigma^{v^{-1}}(C_{n+1,j})\subseteq C_{n,i}\}|.
 $$
We say that $M_{n}$ is the {\it incidence matrix} between $\P_{n}$ and $\P_{n+1}$.  Observe that
$$
\sum_{i=1}^{k_n}M_n(i,j)=|J_n|=\frac{|F_{n+1}|}{|F_n|}, \mbox{ for every } 1\leq j \leq k_{n+1}.
$$

Using the terminology introduced in \cite[Section 5]{CP}, this implies that the sequence $(M_n)_{n\geq 0}$ is {\it managed} by $(|F_n|)_{n\geq 0}$.  That is:
\begin{enumerate}
\item $M_n$ has $k_n\geq 2$ rows and $k_{n+1}\geq 2$ columns;

\item $\sum_{i=1}^{k_n}M_n(i,k)=\frac{|F_{n+1}|}{|F_n|}$, for every
$1\leq k\leq k_{n+1}$.

\end{enumerate}
It is  easy to check that  $M_n(\triangle(k_{n+1}, |F_{n+1}|))\subseteq \triangle(k_n, |F_n|)$, where
$$\triangle(k,p)=\left \{(x_1,\cdots, x_k)\in (\RR^+)^k: \sum_{i=1}^kx_i=\frac{1}{p}\right \}.$$  Thus the following inverse limit is well defined.
$$
\varprojlim_n(\triangle(k_n, |F_n|),M_n)=\{(z_n)_{n\geq 0}\in \prod_{n\geq 0}\triangle(k_n, |F_n|): z_n=M_nz_{n+1} \forall n\geq 0\}.
$$

\begin{remark}\label{strictly-positive}
{\rm  We can assume that the matrices $M_n$ are strictly positive. Indeed, if there exists $B_{n,k}$ such that for every $m>n$ it does not appear in $B_{m,1}$, then the clopen set $\{x\in X: x(F_n)=B_{n,k}\}$  is empty. Thus we can assume that for every $n\geq 0$ and $1\leq k\leq k_n$ there exists $m_{n,k}$ such that $B_{n,k}$ appears (as a translated copy) in $B_{m_{n,k},1}$. By (C1) we can assume that   $m_{n,k}=m_n$ is independent on $k$. Thus and by (C1) again, the product $M_n\cdots M_{m_n+1}$ is strictly positive.}
\end{remark}

The next Lemma is the key to show that $\M(X,\sigma|_X, G)$  is affine homeomorphic to $\varprojlim_n(\triangle(k_n, |F_n|),M_n)$.
 
\begin{lemma}\label{border}   $X\setminus \partial X$ is a full measure set of $X$.
\end{lemma}
\begin{proof}
Let $\mu\in \M(X,\sigma|_X, G), g\in G$ and $n\geq 0$.  We have
\begin{eqnarray*}
\mu\left( \bigcup_{k=1}^{k_n}\bigcup_{v\in F_n\setminus F_ng}\sigma^{v^{-1}}(C_{n,k})  \right) & = &|F_n\setminus F_ng| \sum_{k=1}^{k_n}\mu(C_{n,k})\\
    &=&\frac{|F_n\setminus F_ng|}{|F_n|}
\end{eqnarray*}
Then
$$
\mu\left( \bigcap_{n\geq 0}\bigcup_{k=1}^{k_n}\bigcup_{v\in F_n\setminus F_ng}\sigma^{v^{-1}}(C_{n,k})  \right)\leq \lim_{m\to \infty}\frac{|F_m\setminus F_mg|}{|F_m|}=0,
$$
which implies that $\partial X$  has zero measure with respect to any invariant measure of $(X, \sigma|_X, G)$.
\end{proof}

From Proposition \ref{minimal-free} and Lemma \ref{border} we get: 
\begin{coro}
$(X, \sigma|_X, G)$ is free on a full measure set.
\end{coro}

\begin{proposition}\label{Choquet-measures}
 There is an affine homeomorphism between   $\M(X,\sigma|_{X},G)$ and the
inverse limit  $\varprojlim_n(\triangle(k_n, |F_n|),M_n)$. 
\end{proposition}

\begin{proof} From Lemma \ref{border}, the invariant measures of $(X, \sigma|_X, G)$ are supported on $X\setminus \partial X$, and every point in this set is separated by the atoms  of the partitions $\P_n$'s. Thus every open set $U\subseteq X$ is a (countable) union of elements of the atoms of the partitions $\P_n$'s and a set in $\partial X$. This implies that the measure of $U$ is completely determined by the measures of the atoms in $\P_n$'s.  The rest of the proof follows according to \cite[Proposition 2]{CP}.
\end{proof}

\begin{remark} 

{\rm  For any free minimal Cantor system $(X, T, G)$, denote by $D_m(X, T, G)$ the quotient of the additive group $C(X,\ZZ)$ by the subgroup $\{f\in C(X,\ZZ): \int f d\mu = 0, \forall \mu\in \M(X, T,G)\}$, by $D_m(X, T, G)^+$ the positive cone of the classes $[f]$ in $D_m(X, T, G)$ of the functions $f\geq 0$, and by $1$ the class of the constant function $1$. It is known that the ordered group with unit $\G(X,T,G)=(D_m(X, T, G), D_m(X, T, G)^+,1)$ is invariant under topological orbit equivalence (see \cite{GMPS10}). 
In our case, it is not difficult to show that the group generated by the classes of the functions $1_{C_{n,k}}$'s in $D_m(X,\sigma|_X,G)$  is isomorphic, with the induced order and unit, to the dimension group $(H, H^+, u)$, where
$$
H=\xymatrix{{\ZZ^{}}\ar@{->}[r]^{M^T} &{\ZZ^{k_0}}\ar@{->}[r]^{M^T_0} & {\ZZ^{k_1}}\ar@{->}[r]^{M_1^T} &
{\ZZ^{k_2}}\ar@{->}[r]^{M^T_2}& {\cdots}},
$$
$M=|F_0|(1,\cdots,1)$ , $H^+$ is the usual positive cone, and $u=[M^T,0]$.  

Since we do not know if the atoms of the partitions $\P_n's$ generate the topology of $X$, it is unclear for us if $(H, H^+,u)$ is isomorphic  to  $\G(X,\sigma|_X,G)$. Nevertheless, Proposition \ref{Choquet-measures} and \cite[Lemma 1]{CP} ensure that both groups have the same space of traces. 

On the other hand, observe that the rational subgroup of  $\G(X, \sigma|_X, G)$ contains the subgroup
$\langle \{\frac{1}{|F_n|}: n\geq 0\}\rangle,$
which shows it is non cyclic, as is the case for Toeplitz subshifts. }
\end{remark}




\section{Proof of the principal results.}\label{section-results}

The following result corresponds to part (iii) of \cite[Lemma 8]{CP}

\begin{lemma}\label{lema8}
Let $(M_n)_{n\geq 0}$ be a sequence of matrices which is  managed by $(|F_n|)_{n\geq 0}$. For every $n\geq 0$, we denote by $k_n$ the number of rows of $M_n$. Suppose there exists a constant $K>0$ such that
$$
k_{n+1}\leq K\frac{|F_{n+1}|}{|F_n|}, \mbox{ for every } n\geq 0.
$$
Then there exists an increasing sequence $(n_i)_{i\geq 0}$ in $\ZZ^+$ such that for every $i\geq 0$ and for every $1\leq k\leq k_{n_{i+1}}$, 
$$
k_{n_{i+1}}<M_{n_i}\cdots M_{n_{i+1}-1}(l,k) \mbox{ for every } 1\leq l\leq k_{n_i}.
$$
\end{lemma}

 The proof of the next proposition is similar to   \cite[Proposition 3]{CP}.
\begin{proposition}\label{construction}
Let $(M_n)_{n\geq 0}$ be a sequence of matrices which is  managed by $(|F_n|)_{n\geq 0}$. For every $n\geq 0$, we denote by $k_n$ the number of rows of $M_n$. Suppose there exists $K>0$ such that $k_{n+1}\leq K\frac{|F_{n+1}|}{|F_n|}$, for every $n\geq 0$. Then there exists a minimal free $G$-subshift $(X, \sigma|_X, G)$ such that 
 $\M(X,\sigma|_{X},G)$ is affine homeomorphic to
inverse limit  $\varprojlim_n(\triangle(k_n, |F_n|),M_n)$.
 \end{proposition}
\begin{proof}
Because of Lemma \ref{lema8}, we can assume that for every $n\geq 0$,
$$
k_{n+1}<\min\{M_n(i,j): 1\leq i\leq k_n, 1\leq j\leq k_{n+1}\}.
$$ 
For every $n\geq 0$, let $\tilde{M}_n$ the $(k_n+1)\times (k_{n+1}+1)$ dimensional matrix defined as
$$
\tilde{M}_n(\cdot, 1)=\tilde{M}_n(\cdot, 2)=\left(\begin{array}{l}
                                              1\\
                                               M_n(1,1)-1\\
                                               M_n(2,1)\\
                                               \vdots\\
                                               M_n(k_n,1)
                                               \end{array}\right) 
$$, and
$$
\tilde{M}_n(\cdot, k+1)= \left(\begin{array}{l}
                                              1\\
                                               M_n(1,k)-1\\
                                               M_n(2,k)\\
                                               \vdots\\
                                               M_n(k_n,k)
                                               \end{array}\right), \mbox{ for every } 2\leq k\leq k_{n+1}.  
$$
From \cite[Lemma 1]{CP} and \cite[Lemma 2]{CP} we have that the inverse limits $\varprojlim_n(\triangle(k_n, |F_n|),M_n)$ and $\varprojlim_n(\triangle(k_n+1, |F_n|), \tilde{M}_n)$ are affine homeomorphic. 
 Observe that  $(\tilde{M}_n)_{\geq 0}$ is managed by $(|F_n|)_{n\geq 0}$ and verifies
for every $n\geq 0$,
\begin{equation}\label{importante}
3\leq k_{n+1}+1\leq \min\{M_n(i,j): 2\leq i\leq k_n+1, 1\leq j\leq k_{n+1}+1\}.
\end{equation}
In order to reduce the notation, we call $l_n$ and $l_{n+1}$ the number of rows and columns of $\tilde{M}_n$ respectively.

Let  $\Sigma=\{0,\cdots, l_0\}$. For every $1\leq k\leq k_0$ let define $B_{0,k}\in \Sigma^{F_0}$ as
$$
B_{0,k}(v)= \left\{\begin{array}{ll}
                                            k  &  \mbox{ if } v=1_G\\
                                            0  & \mbox{ if } v\in F_0\setminus\{1_G\}.
                                            \end{array}\right. 
$$ 

For $n\geq 0$, suppose that we have defined $B_{n,1}, \cdots, B_{n,l_n}$ different elements in $\Sigma^{F_n}$ that satisfy condition (C3). Observe this is true for $n=0$.  We define $B_{n+1, 1},\cdots, B_{n+1, l_{n+1}}$ in $\Sigma^{F_{n+1}}$ as follows: for every $1\leq k \leq k_{n+1}$
$$B_{n+1,k}(F_n)=B_{n,1},$$ 
and for every  $c\in J_n\setminus \{1_G\}$,
$$B_{n+1,k}(cF_n)\in \{B_{n,2},\cdots, B_{n,k_n}\}$$  in a way such that
$$|\{v\in J_n: B_{n+1,k}(cF_n)=B_{n,i}\}|=\tilde{M}_n(i,k),$$
for every $2\leq i\leq l_n$.

Condition (\ref{importante}) ensures that it is possible to make this procedure in order that $B_{n+1,k}\neq B_{n+1,s}$ if $k\neq s$.  By construction, $B_{n,1},\cdots, B_{n,l_n}$ satisfy (C1), (C2) and (C3). Then, the associated system $(X, \sigma|_X, G)$ is a minimal  free (see Proposition \ref{minimal-free}) $G$-subshift such that $\M(X, \sigma|_X, G)$ is affine homeomorphic to  $\varprojlim_n(\triangle(k_n, |F_n|),M_n)$ (see Proposition \ref{Choquet-measures}).
\end{proof}

\begin{proof}[Proof of Theorem 2]
The proof is direct from Proposition \ref{construction}, \cite[Lemma 9]{CP} and \cite[Lemma 12]{CP}.
\end{proof}
 
 
 The next result is a consequence of Theorem \ref{nilpotentes} and Theorem \ref{theoremA}.
 
  \begin{coro}
  Let $G$ be a countable infinite nilpotent group. For every metrizable Choquet simplex $K$, there  exists a minimal    $G$-subshift which is free on a full measure set,   whose set of invariant probability measures is affine homeomorphic to $K$. If $G$ is abelian, then the subshift is free.
  \end{coro}
 
  \medskip

{\bf Acknowledgments.}  We would like to  thank  Brandon Seward  for  
his valuable comments about monotilable groups.


\end{document}